\documentclass{amsproc}
\usepackage{amssymb}
\usepackage{amsfonts}

\theoremstyle{plain}

\newtheorem{corollary}{Corollary}

\newtheorem{theorem}{Theorem}
\numberwithin{equation}{section}
\input{tcilatex}

\begin{document}
\title{On Chevalley's Extension Theorem}
\author{Muhammad Zafrullah}
\address{Department of Mathematics, Idaho State University, Pocatello, 83209
ID}
\email{mzafrullah@usa.net}
\urladdr{}
\dedicatory{Dedicated ti the memory of Paul Cohn}

\begin{abstract}
Professor Daniel Anderson informed me, recently, that there is an error in
the proof of Theorem 56 of Kaplansky's book on Commutative Rings. His
(Dan's) reason was "He (Kaplansky) orders by reverse inclusion but in the
last line uses inclusion, so we don't contradict maximality (which is
minimality)". The aim of this short note is to indicate that while Dan
Anderson appears to be correct in pointing out an error in the proof of
Theorem 56 of \cite{K}, the statement of the theorem is a correct
consequence of a Theorem of Chevalley's.
\end{abstract}

\maketitle

\bigskip Professor Daniel Anderson informed me, recently, via \cite{A}, that
there is an error in the proof of Theorem 56 of Kaplansky's book on
Commutative Rings. His (Dan's) reason was "He (Kaplansky) orders by reverse
inclusion but in the last line uses inclusion, so we don't contradict
maximality (which is minimality)". Looking at the theorem and its proof, I
realized that I had seen a similar result elsewhere. After some search I
found Chevalley's Extension Theorem as Theorem 3.1.1 of \cite{EP}. The aim
of this short note is to indicate that while Dan Anderson appears to be
correct in pointing out an error in the proof of Theorem 56 of \cite{K}, the
Theorem is correct as stated as it follows from a theorem of Chevalley's. We
include, below, Chevalley's theorem and its proof to indicate how it is
related to another theorem in \cite{K}.

It is hard to believe that there would be such an error in \cite{K}, but Dan
Anderson is a very serious and respected Mathematician, and a student of
Kaplansky's. Consider this another reason why this note got written.

I give the statement and a redo of the proof, below, of Chevalley's
Extension Theorem.

\begin{theorem}
\label{Theorem A}Given that $K$ is a field, let $R\subseteq K$ be a subring
of $K$ and let $P\subseteq R$ be a prime ideal of $R.$ Then there exists a
valuation ring $O$ of $K$ such that $R\subseteq O$ and $M\cap R=P,$ where $M$
is the maximal ideal of $O.$
\end{theorem}

\begin{proof}
We use the standard notation $R_{P}$ for localization of $R$ at $P.$ Let $%
\sum =\{(A,I)|R_{P}\subseteq A\subseteq K$, $pR_{P}\subseteq I\subseteq A\}$
where $A$ is a ring and $I$ a proper ideal of $A$. Then $\sum \neq \phi ,$
because $(R_{P},pR_{P})\in \sum .$ Moreover $\sum $ may be partially ordered
as follows: for all $(A_{j},I_{j})$ $\in \sum $, ($j=1,2)$ we declare $%
(A_{1},I_{1})\leq (A_{2},$ $I_{2})\Leftrightarrow A_{1}\subseteq A_{2}$ and $%
I_{1}\subseteq I_{2}.$ For each chain $\{(A_{j},I_{j})|j\in J$ where $J$ is
an index set$\}$ we have an upper bound $(\cup A_{j},\cup I_{j})\in (\sum
,\leq ).$ By Zorn's Lemma, $\sum $ has a maximal element $(O,M).$ Observe
that $R\subseteq R_{P}\subseteq O,$ and since $PR_{P}$ is the maximal ideal
of $R_{P}$ we have $M\cap R_{P}=PR_{P}$ and consequently $M\cap R=P.$ So, to
complete the proof, it remains to show that $(O,M)$ is a valuation domain.
From the maximality of $(O,M)$ we first conclude that $O$ is a local ring.
Assume now that $O$ is not a valuation ring. Then there is $x\in K\backslash
\{0\}$ such that $x,x^{-1}\notin O.$ But then $O\subsetneq O[x],$ $O[x^{-1}]$%
. The maximality of $(O,M)$ implies therefore that $M[X]=O[x]$ and $%
M[x^{-1}]=O[x^{-1}].$ But then there exist $a_{0},...,a_{n};b_{0},...,b_{m}%
\in M$ such that

$1=\sum_{i=0}^{n}a_{i}x^{i}$ and $1=\sum_{i=0}^{m}b_{i}x^{-i}$ ....(i) with $%
n,m$ minimal.

Suppose for a start, that $m\leq n.$ As $b_{0}\in M,$ we have $%
\sum_{i=1}^{m}b_{i}x^{-i}=1-b_{0}\in O\backslash M$ (a nonzero non unit).
Or, dividing both sides of the previous equation by $1-b_{0}$ we get, $%
\sum_{i=1}^{m}\frac{b_{i}}{1-b_{0}}x^{-i}=1$. Thus we have $%
\sum_{i=1}^{m}c_{i}x^{-i}=1$ ..... (ii) where $c_{i}=\frac{b_{i}}{1-b_{0}}.$

Multiplying both sides of (ii) by $x^{n}$ we get $%
\sum_{i=1}^{m}c_{i}x^{n-i}=x^{n}$....(iii)

Now from (i) we have $1=\sum_{i=0}^{n}a_{i}x^{i}=%
\sum_{i=0}^{n-1}a_{i}x^{i}+a_{n}x^{n}$ ....(iv)

Substituting in (iv) the value of $x^{n}$ from (iii) we get

$1=\sum_{i=0}^{n-1}a_{i}x^{i}+a_{n}\sum_{i=1}^{m}c_{i}x^{n-i}$ ....(v)

Because $m\leq n,$ powers $p$ of $x$ in each summand on the right of (v) are 
$0\leq p\leq $ $n-1.$ But this contradicts the minimality of $n$ in
expressing $1$ as a polynomial in $x.$ If, on the other hand, we take $n\leq
m,$ then arguing in a similar fashion, we get a contradiction to the
minimality of $m.$
\end{proof}

Let $R\subseteq S$ be an extension of rings and let $I$ be a proper ideal of 
$R.$ Let us say that $I$ survives in $S$ if $I$ generates a proper ideal of $%
S,$ i.e., if $IS\neq S.$

\begin{corollary}
\label{Corollary B} (Kaplansky Theorem 56). Let $K$ be a field, $R$ a
subring of $K$, and $I$ an ideal in $R,I\neq R$. Then there exists a
valuation domain $V$, $R$ $\subseteq $ $V\subseteq K$, such that $K$ is the
quotientfield of $V$ and $I$ survives in $V$.
\end{corollary}

\begin{proof}
Because $I\neq R$ there is a prime ideal $P$ of $R$ such that $I\subseteq P.$
Now by Theorem \ref{Theorem A} there is a valuation domain $(V,M)$ such that 
$P=M\cap R,$ i.e, $P$ survives in $V$ and consequently $I$ survives in $V.$
\end{proof}

The proof of Theorem \ref{Theorem A} can be slightly modified to produce
another interesting corollary.

\begin{corollary}
\label{Corollary C} Let $K$ be a field and $R$ a subring of $K$. Let $u\in
K\backslash \{0\}$, and let $I$ be an ideal in $R$, $I\neq R$. Then $I$
survives either in $R[u]$ or in $R[u^{-1}]$.
\end{corollary}

\begin{proof}
Suppose that $I$ survives in neither. Then $IR[u]=R[u]$ and $%
IR[u^{-1}]=R[u^{-1}].$ Then there exist $a_{0},...,a_{n};b_{0},...,b_{m}\in
I $ such that

$1=\sum_{i=0}^{n}a_{i}x^{i}$ and $1=\sum_{i=0}^{m}b_{i}x^{-i}$ ....(I) with $%
n,m$ minimal.

From the second expression in (I) we have $\sum_{i=1}^{m}b_{i}x^{-i}=1-b_{0}$
... (II)

Assuning that $m\leq n$ and multiplying (II) by $x^{n},$ throughout, we get $%
\sum_{i=1}^{m}b_{i}x^{n-i}=(1-b_{0})x^{n}$ ...(III)

Next, multiplying the first equation in (I) by $(1-b_{0})$ we have

$(1-b_{0})=\sum_{i=0}^{n-1}a_{i}(1-b_{0})x^{i}+a_{n}(1-b_{0})x^{n}$ ... (IV)

Now, substituting the value of $(1-b_{0})x^{n}$ in (IV) and re-writing we get

$1=b_{0}+\sum_{i=0}^{n-1}a_{i}(1-b_{0})x^{i}+a_{n}\sum_{i=1}^{m}b_{i}x^{n-i}$
...(V)

Since every power of $x$ that appears in (V) is less than $n$ we conclude
that $1$ can be expressed as a polynomial of degree less than $n$ and that
contradicts the minimality of $n.$ Finally assuming that $n\leq m$ and
reversing the roles of $m$ and $n$ in the above calculations we get a
similar contradiction.
\end{proof}

Now let us change the wording of Corollary \ref{Corollary C} to see that
with the same wording as in the proof of Corollary \ref{Corollary C} we can
prove.

\begin{corollary}
\label{Corollary D} Let $R\subseteq T$ be rings, let $u$ be a unit in $T$,
and let $I$ be an ideal in $R$, $I\neq R$. Then $I$ survives either in $R[u]$
or in $R[u^{-1}]$.
\end{corollary}

Observe that Corollary \ref{Corollary D} is precisely Theorem 55 of \cite{K}.

Finally, thanks to Kaplansky's students and disciples Chevalley's Extension
Theorem gets cited a lot, in the form of Theorem 56 of \cite{K}, in
Multiplicative Ideal Theory, and the paper \cite{DHLZ} is no exception. Now
if there is a comment about the veracity of Theorem 56 of \cite{K}, from a
big gun like Dan Anderson, it would seriously undermine the confidence in
all the papers using that theorem, with \cite{DHLZ} included and that is
another reason for jotting down the above few lines. I hope I have been able
to establish the veracity of the statement of Theorem 56 of \cite{K}. Of
course if the ordering is reversed in the proof of Theorem 56 of \cite{K},
to fit Dan's requirement, then the proof will become all right, but then it
would clearly appear to have been taken from Chevalley's Extension Theorem!

In addition to \cite{EP}, one can find Chevalley's Theorem, without a
reference to Chevalley, as Lemma 4.3 in \cite{C A2} with statement:

\begin{theorem}
\label{Theorem E} Let $K$ be a field, $R$ a subring of $K$ and $\mathfrak{a}$
a nonzero proper ideal in $R.$ Then there is a proper subring $V$ with an
ideal $\mathfrak{p}$ such that $(V,$ $\mathfrak{p)}$ is maximal among pairs
dominating $(R,$ $\mathfrak{a}).$ Further, any such maximal pair $(V,%
\mathfrak{p)}$ consists of a valuation ring $V\neq K$ and its maximal ideal $%
\mathfrak{p}$.
\end{theorem}

Here a pair $(R,$ $\mathfrak{a})$, where $R$ is a subring of $K$ and $%
\mathfrak{a}$ an ideal of $R$ is said to dominate another pair $(R^{\prime
}, $ $\mathfrak{a}^{\prime }),$ if $R\supseteq R^{\prime }$ and $\mathfrak{%
a\supseteq a}^{\prime }.$ In \cite{C BA} Theorem \ref{Theorem E} appears as
Lemma 9.4.3, with exactly the statement as above and with a proof similar to
that of Theorem \ref{Theorem A}. Yet in this book Paul does mention
Chevalley, including \cite{Chev} as a reference at the end of the book.
Theorem 1 in \cite{Chev} comes close to what is called Chevalley's Extension
Theorem in \cite{EP}. Though couched in a totally different language, the
proof has a similar flavor. If there is any doubt, Remark 1 after the proof
removes it by saying that Neither the definition of a $V$-ring in $R$, nor
the proof of Theorem 1, makes any use of the fact that $R$ is a field of
algebraic functions of one variable. It follows that our proof of Theorem 1
yield a result which is valid for any pair of fields $(K,R)$ such that $K$
is a subfield of $R.$ (For our purposes, a $V$-ring is a valuation domain $%
\mathcal{V}$ and a place is its maximal ideal.)

Finally, there has been a spate of some clever Mathematicians taking the
easy way out by presenting "new results" by a mere change of terminology and
adopting the results they fancy, from someone who they think may not have a
voice. Here are two of the examples that I am painfully aware of:
https://lohar.com/mithelpdesk/hd2004.pdf and
https://lohar.com/mithelpdesk/hd2006.pdf

In my opinion this cannot happen without the help of a supportive referee
and as this plague seems to be rampant in the so called multiplicative ideal
theory and among the students of Dan Anderson's, I am beginning to see
Kaplansky's Theorems 55 and 56 of \cite{K}, in a similar light, a heist from
an unsuspecting fellow's work.

\bigskip

\end{document}